\documentclass[12pt]{amsart}
\usepackage{amssymb,amscd}
\usepackage{amsmath}

\begin{document}
\newtheorem{thm}{Theorem}[section]
\newtheorem{lem}[thm]{Lemma}
\newtheorem{prop}[thm]{Proposition}
\newtheorem{cor}[thm]{Corollary}
\theoremstyle{definition}
\newtheorem{ex}[thm]{Example}
\newtheorem{rem}[thm]{Remark}
\newtheorem{prob}[thm]{Problem}
\newtheorem{thmA}{Theorem}
\renewcommand{\thethmA}{}
\newtheorem{defi}[thm]{Definition}
\renewcommand{\thedefi}{}
\input amssym.def
\long\def\alert#1{\smallskip{\hskip\parindent\vrule%
\vbox{\advance\hsize-2\parindent\hrule\smallskip\parindent.4\parindent%
\narrower\noindent#1\smallskip\hrule}\vrule\hfill}\smallskip}
\def\ff{\frak}
\def\Spec{\mbox{\rm Spec}}
\def\type{\mbox{ type}}
\def\Hom{\mbox{ Hom}}
\def\rank{\mbox{ rank}}
\def\Ext{\mbox{ Ext}}
\def\Ker{\mbox{ Ker}}
\def\Max{\mbox{\rm Max}}
\def\End{\mbox{\rm End}}
\def\l{\langle\:}
\def\r{\:\rangle}
\def\Rad{\mbox{\rm Rad}}
\def\Zar{\mbox{\rm Zar}}
\def\Supp{\mbox{\rm Supp}}
\def\Rep{\mbox{\rm Rep}}
\def\cal{\mathcal}
\title[a non-commutative generalization of \L ukasiewicz rings]{a non-commutative generalization of \L ukasiewicz rings}
\thanks{\today}
\author{Albert Kadji, Celestin Lele, Jean B. Nganou}
\address{Department of Mathematics, University of Yaounde I,
Cameroon} \email{kadjialbert@yahoo.fr}
\address{Department of Mathematics and Computer Science, University of Dschang, Cameroon
} \email{celestinlele@yahoo.com}
\address{Department of Mathematics, University of Oregon, Eugene,
OR 97403} \email{nganou@uoregon.edu}
\begin{abstract} The goal of the present article is to extend the study of commutative rings whose ideals form an MV-algebra as carried out by Belluce and Di Nola \cite{BN} to non-commutative rings. We study and characterize all rings whose ideals form a pseudo MV-algebra, which shall be called here generalized \L ukasiewicz rings. We obtain that these are (up to isomorphism) exactly the direct sums of unitary special primary rings.
\vspace{0.20in}\\
{\noindent} Key words: MV-algebra, \L ukasiewicz ring, QF-ring, pseudo MV-algebra, semi-ring, special primary ring, Dubrovin valuation ring, Brown-McCoy radical, Jacobson radical.
\end{abstract}
\maketitle
\section{Introduction}
A ring $R$ is said to be generated by central idempotents, if for every $x\in R$, there exists a central idempotent element $e\in R$ such that $ex=x$. The (two-sided) ideals of such a ring $R$ form a residuated lattice $A(R):=\langle \text{Id}(R), \wedge, \vee, \odot, \to, \rightsquigarrow, \{0\}, R\rangle $, where $$I\wedge J = I\cap J,
I\vee J=I+J, I \odot J:=I\cdot J,$$
$$I \to J := \{x\in R: Ix\subseteq J \}, I\rightsquigarrow J :=
\{x\in R:xI\subseteq J \}.$$
\par Of these rings, Belluce and Di Nola \cite{BN} investigated the commutative rings $R$ for which $A(R)$ is an MV-algebra, which they called \L ukasiewicz rings. Recall that MV-algebras, which constitute the algebraic counterpart of \L ukasiewicz many value logic are categorically equivalent to Abelian $\ell$-groups with strong units \cite{CM}. An important non-commutative generalization of MV-algebra, known as pseudo MV-algebra was introduced by Georgescu and Iorgulescu \cite{g1, g}. These have been studied extensively (see for e.g., \cite{CK, ad, a, DP}).\par
The natural question that arises is what happens if one drops the commutativity assumption on \L ukasiewicz rings. One would expect the residuated lattice $A(R)$ above to become a pseudo MV-algebra. One embarks then on the study of  generalized \L ukasiewicz rings (referred to in the article as GLRs, for short), which are rings (both commutative and non-commutative) for which $A(R) $ is a pseudo MV-algebra. 
 \par
The main goal of this work is to completely characterize the rings $R$ for which $A(R)$ is a pseudo MV-algebra. From the onset, the requirements on these rings appear to be very restrictive. However, it is quite remarkable that this class includes some very important classes of rings such as left Artinian chain rings, some special factors of Dubrovin valuation rings, and matrix rings over \L ukasiewicz rings. \par
The paper is organized as follows. In section 2, we introduce and study generalized \L ukasiewicz semi-rings, which comprised the \L ukasiewicz semi-rings as studied in \cite{BN}. We show that these are dually equivalent to pseudo MV-algebras. In section 3, we introduce and study the main properties of generalized \L ukasiewicz rings. In particular, we show that they are closed under finite direct products, quotients by ideals, and direct sums. In section 4, we prove a representation theorem for generalized \L ukasiewicz rings. We obtain that (up to isomorphism), generalized \L ukasiewicz rings are direct sums of unitary special primary rings.\par
In the paper, when the term ideal is used, it shall refer to two-sided ideal.\par
A pseudo-MV algebra can be defined as an algebra  $A=\langle A, \oplus, {}^-,
{}^\sim, 0, 1 \rangle$  of type $(2, 1, 1, 0, 0)$ such that the
following axioms hold for all $x,y,z \in  A $ with an additional operation $x\odot y=(y^-\oplus x^-)^\sim$
\begin{itemize}
\item[(A1)] $(x\oplus y)\oplus z = x\oplus (y \oplus z) $;
\item[(A2)] $ x\oplus 0 = 0 \oplus x = x $;
\item[(A3)]$x\oplus 1 = 1 \oplus x = 1$;
\item[(A4)]$ {1}^\sim =0 ,  {1}^- = 0 $;
\item[(A5)] $( {x}^- \oplus {y}^-)^\sim = ( {x}^\sim \oplus {y}^\sim)^- $;
\item[(A6)] $ x \oplus ( {x}^\sim \odot y) = y \oplus ( {y}^\sim \odot x)=
(x \odot {y}^- ) \oplus y =  (  y \odot {x}^- ) \oplus x $;
\item[(A7)] $ x \odot ( {x}^-  \oplus   y) = (x \oplus  {y}^\sim) \odot y $;
\item[(A8)] $   {({x}^-)}^\sim =x .$
\end{itemize}
\par Every pseudo MV-algebra has an underline distributive lattice structure, where the order $\leq$ is defined by: $$x\leq y \; \; \text{ if and only if}\; \; x^-\oplus y=1.$$ Moreover, the infimum and supremum in this order are given by: 
\begin{center}
(i)\  $x \vee y=  x \oplus ( {x}^\sim \odot y) = y \oplus ( {y}^\sim
\odot x)= (x \odot {y}^- ) \oplus y =  (  y \odot {x}^- ) \oplus
x $,\\
(ii)\  $x \wedge y=  x \odot ( {x}^- \oplus y) = y \odot ( {y}^-
\oplus x)= (  x \oplus {y}^\sim ) \odot y =  (  y \oplus {x}^\sim )
\odot x $.
\end{center}
The prototype of pseudo MV-algebra can be constructed from an $\ell$-group as follows. Let $G$ be an $\ell$-group and $u$ a positive element in $G$, then $\langle \Gamma (G,u), \oplus, {}^-, {}^\sim, 0, u\rangle$, where
$$\Gamma(G,u):=\{x\in G:0\leq x\leq u\},$$
$$x\oplus y:=(x+y)\wedge u,\; \; \; \; \; \; x^\sim:=-x+u,$$
$$x\odot y=(x-u+y)\vee 0,\; \; \; \; \; \; x^-:=u-x,$$
is a pseudo MV-algebra.\par
In fact, a remarkable result due to Dvure\v{c}enskij \cite{a} asserts that every pseudo MV-algebra is isomorphic to a pseudo MV-algebra of the form $\langle \Gamma (G,u), \oplus, {}^-, {}^\sim, 0, u\rangle$.
Pseudo MV-algebras have a wealth of properties that will be used repeatedly without any explicit citation. Most of the properties used can be found in  \cite{CK, ad, a, DP}. 
\section{Semi-rings and pseudo MV-algebras}
In this section, we introduce generalized \L ukasiewicz semi-rings and study their connections to pseudo MV-algebras.
\begin{prop}
Let $A=\langle A, \oplus, \odot, {}^-, {}^\sim , 0, 1 \rangle$ be a pseudo MV-algebra and $S(A):=\langle A,+, \cdot, 0, 1 \rangle$. Then $S(A)$ is an additively idempotent semi-ring satisfying:\\
(i) $x\cdot y=0$ iff $y\leq {x}^-$ iff $x\leq {y}^\sim $;\\
(ii) $x+y= {({({x}^\sim \cdot y)}^\sim \cdot {x}^\sim)}^-
 = {( {x}^\sim
 \cdot {({y} \cdot {x}^-)}^\sim)}^-$;\\
 (iii)  ${({y}^\sim \cdot {x}^\sim)}^- = { ({y}^-\cdot {x}^-)}^\sim$;\\
 where $x+y=x\vee y$, $x\cdot y=x\odot y$, and $x\leq y$ iff ${x}^-\oplus y=1$.
 \end{prop}
 \begin{proof}
 Follows easily from the main properties of pseudo MV-algebras.  \end{proof}
The construction above can be reversed as we will now proceed to justify.\\
Let $S=\langle S, +,  \cdot, 0, 1 \rangle$ be an additively
idempotent semi-ring, i.e., $x+x=x$ for all $x\in S$.\\ Define the relation $\leq$ on $S$ by $x\leq y$ if and only if $x+y=y$.\\
 $S$ is called a  generalized $\L$ukasiewicz semi-ring if there exists maps ${}^-:S\to S$  and ${}^\sim:S\to S$  satisfying for all $x,y\in S$:\\
(i) $x\cdot y=0$ iff $y\leq  {x}^-$ iff $x\leq {y}^\sim $;\\
(ii) $ x + y=  {({({x}^\sim \cdot y)}^\sim \cdot {x}^\sim)}^- = {( {x}^\sim \cdot {({y} \cdot {x}^-)}^\sim)}^- $;\\
(iii)  ${({y}^\sim \cdot {x}^\sim)}^- = { ({y}^-\cdot {x}^-)}^\sim $.
 \begin{lem}\label{semi}
 Let $S=\langle S, +,  \cdot,^-, ^\sim, 0, 1 \rangle$ be a generalized \L ukasiewicz semi-ring and $\leq$ the relation defined above. Then each of the following properties holds for every $x, y\in S$.\\
 (i) The relation $\leq$ is an order relation on $S$ that is compatible with $+$ and $\cdot$;\\
 (ii)  $x^\sim \cdot x=x\cdot x^-=0$, $0^\sim=0^-=1$ and $1^\sim=1^-=0$.\\
 (iii) $x \leq y$ implies $y^\sim\leq x^\sim$ and $y^-\leq x^-$.\\
 (iv) $ {{x}^-}^\sim ={{x}^\sim}^-=x$.\\
 (v) $S$ is a lattice-ordered semi-ring, where $x\vee y=  x + y$ and $x\wedge y= {( {x}^- + {y}^-)}^\sim = {({x}^\sim +{y}^\sim )}^-$.\\
 (vi) ${( {x}^- + {y}^-)}^\sim = {({x}^\sim +{y}^\sim )}^-$.
 \end{lem}
 \begin{proof}
 (i) $\leq$ is clearly reflexive and anti-symmetric. In addition, if $x+y=y$ and $y+z=z$, then $x+z=x+y+z=y+z=z$. Thus $\leq$ is transitive. The compatibility of $\leq$ with $+$ and $\cdot$ is also easy to verify.\\
 For the rest of the properties, note that combining $(ii)$ and $(iii)$, one obtains that the following property holds in any generalized \L ukasiewicz semi-ring.
$$(ii)' \; \; \;  x + y= {({({x}^\sim \cdot y)}^- \cdot {x}^-)}^\sim = {( {x}^- \cdot {({y} \cdot {x}^-)}^-)}^\sim.$$
 (ii) Since $x^\sim \leq x^\sim $ and $x^-\leq x^-$, then it follows from (i) of the definition of a generalized \L ukasiewicz semi-ring that  $x^\sim \cdot x=x\cdot x^-=0$. In addition, $1^\sim=1^\sim\cdot 1=0$ and $1^-=1\cdot 1^-=0$. Also, $1=1+1= {({({1}^\sim \cdot 1)}^\sim \cdot {1}^\sim)}^-={({(0 \cdot 1)}^\sim \cdot 0})^-=(0^\sim\cdot 0)^-=0^-$. Similarly, it follows from $(ii)'$ that $0^\sim=1.$\\
 (iii) Suppose that $x\leq y$, then by (i) above, $x\cdot y^-\leq y\cdot y^-=0$ and $y^\sim\cdot x\leq y^\sim\cdot y=0$. Hence, $x\cdot y^-=y^\sim\cdot x=0$, and it follows from (i) of the definition of the generalized \L ukasiewicz semi-ring that $y^\sim\leq x^\sim$ and $y^-\leq x^-$.\\
 (iv) $x=x+x={({({x}^\sim \cdot x)}^\sim \cdot {x}^\sim)}^-=(0^\sim\cdot x^\sim)^-={{x}^\sim}^-$. Using $(ii)'$, a similar verification shows that $x={{x}^-}^\sim$.\\
 (v) It remains to show that with respect to the order $\leq$,  $x\vee y=  x + y$ and $x\wedge y= {( {x}^- + {y}^-)}^\sim = {({x}^\sim +{y}^\sim )}^-$.\\
It is clear that $x+y$ is an upper bound of $x$ and $y$. In addition, suppose that $x, y\leq u$, then $x+y\leq u+u=u$. Hence, $\sup (x,y)=x+y$. \\
We also have $x^-, y^-\leq x^-+y^-$, so $(x^-+y^-)^\sim \leq x, y$. In addition, suppose that $\ell \leq x, y$, then $x^-, y^-\leq \ell^-$, so $x^-+ y^-\leq \ell^-+\ell^-=\ell^-$. Thus, $\ell \leq (x^-+y^-)^\sim$ and $\inf (x, y)=(x^-+y^-)^\sim$. A similar argument shows that $\inf (x, y)={({x}^\sim +{y}^\sim )}^-$.\\
(vi) Clear from (v).
 \end{proof}
 \begin{prop}\label{AS}
 For every $S=\langle S, +,
\cdot,^-, ^\sim, 0, 1 \rangle$, a generalized \L ukasiewicz semi-ring, define $\oplus$ and $\odot$ by $$x\oplus y= {({y}^\sim \cdot {x}^\sim
)}^- \; \; \text{and}\; \; \; x\odot y=x\cdot y$$
Then, $A(S):=\langle S, \oplus,  \odot, {}^-,{}^\sim, 0, 1 \rangle$
is a pseudo MV-algebra.
\end{prop}
\begin{proof} Observe from Lemma \ref{semi} that $x\oplus y= {({y}^\sim \cdot {x}^\sim
)}^-=(y^-\cdot x^-)^\sim$.\\
(A1) $x\oplus (y\oplus z)=x\oplus ((z^\sim\cdot y^\sim)^-)=((z^\sim\cdot y^\sim)\cdot x^\sim)^-=(z^\sim\cdot (y^\sim \cdot x^\sim))^-=(y^\sim \cdot x^\sim)^-\oplus z=(x\oplus y)\oplus z$.\\
(A2), (A3), (A4): follow straight from Lemma \ref{semi}.\\
(A5) Since $x\oplus y= {({y}^\sim \cdot {x}^\sim)}^-=(y^-\cdot x^-)^\sim$, it follows that $(x^-\oplus y^-)^\sim={({x}^\sim \oplus{y}^\sim )}^-=y\cdot x$.\\
(A6) Note that from (ii) of the definition of generalized \L ukasiewicz semi-ring, $x+y=x\oplus (x^\sim \odot y)=(y\odot x^-)\oplus x$. The equalities to the remaining expressions follow from the fact that $+$ is commutative.\\
(A7) Note that $(x+y)^\sim={({x}^\sim \cdot y)}^\sim \cdot {x}^\sim$ and by $(ii)'$ of the proof of Lemma \ref{semi}, we also have $(x+y)^-=x^- \cdot (y \cdot x^-)^-$. Thus, $x\odot (x^-\oplus y)=x\cdot (y^\sim \cdot x)^-=(x^\sim+y^\sim)^-=(x^-+y^-)^\sim=(y^-+x^-)^\sim=(y\cdot x^-)^\sim \cdot y=(x\oplus y^\sim)\cdot y$.\\
(A8) Clear from Lemma \ref{semi} (iv).

  \end{proof}
   Le $S$  be a  generalized $\L$ukasiewicz semi-ring, an ideal of $S$ is a non-empty  subset $I$ closed under $+$ and such that $xy, yx\in I$, whenever $x \in  I $ and $y\in S $.   \begin{prop}\label{equiv-ideal}
 In a  generalized $\L$ukasiewicz semi-ring $S$, the following are
 equivalent:\\
  a) $I$ is and ideal of  $S$;\\
  b) $I$ is a non-empty  subset closed under $+$ and whenever $x \in  I $ and $y \leq  x $, then $y \in  I $ .
  \end{prop}
\begin{proof}
Assume that   $I$ is and ideal of  $S$ and  let  $x \in  I $ and $y\in S$ with $y
\leq x .$ \\
 Then  $ x
 \cdot {({y}^\sim \cdot x)}^- \in  I $ . But  $ x
 \cdot {({y}^\sim \cdot x)}^-  = {( {x}^- + {y}^-)}^\sim ={{y}^-}^\sim =y $. Thus, we obtain that  $y \in  I $.\\
  Conversely , let  $x \in  I $ and $y\in S $, we have  to show that  $xy \in  I $ and $yx \in  I $.
  Since $ 1= y +1  $ , we have $x= xy +x= yx+x $. So $xy \leq  x $
  and $yx \leq  x $ and we obtain that  $xy \in  I $ and $yx \in  I $.
   \end{proof}
  \begin{prop}
There is a natural duality between pseudo MV-algebras and
generalized $\L$ukasiewicz semi-rings.
\end{prop}
\begin{proof}
One needs to prove that $S(A(S))$ and $S$ are equal as generalized \L ukasiewicz semi-rings; and $A(S(A))$ and $A$ are equal as pseudo MV-algebras. Since the underline sets remain unchanged and so do the operations: multiplication, $^-$, and $^\sim$, one only needs to check that the additions coincide. Starting with a generalized \L ukasiewicz semi-ring $S=\langle S, +, \cdot,^-, ^\sim, 0, 1 \rangle$, define $x\oplus y= {({y}^\sim \cdot {x}^\sim)}^-$ and $x\odot y=x\cdot y$. One obtains a pseudo MV-algebra $A(S):=\langle S, \oplus,  \odot, {}^-,{}^\sim, 0, 1 \rangle$, which has a supremum given by $x\vee y=  x \oplus ( {x}^\sim \odot y) = y \oplus ( {y}^\sim \odot x)= (x \odot {y}^- ) \oplus y =  (  y \odot {x}^- ) \oplus x$. Now, from this pseudo MV-algebra, one constructs the generalized \L ukasiewicz semi-ring $S(A(S))$ whose addition is defined as the supremum. Therefore, one only needs to verify that $x\vee y=x+y$, which is clear from (ii) of the definition of a generalized \L ukasiewicz semi-ring. Now, starting with a speudo MV-algebra $A=\langle A, \oplus, \odot, {}^-, {}^\sim , 0, 1 \rangle$, one construct a generalized \L ukasiewicz semi-ring $S(A):=\langle A,+, \cdot, 0, 1 \rangle$, where $x+y=x\vee y$, the supremum of the pseudo MV-algebra and $x\cdot y=x\odot y$. From this generalized \L ukasiewicz semi-ring, one gets a pseudo MV-algebra $A(S(A))$, with $x\oplus' y=(y^\sim\odot x^\sim)^-=(y^-\odot x^-)^\sim$. Therefore, one needs to check that $x\oplus' y=x\oplus y$, that is $x\oplus y=(y^\sim\odot x^\sim)^-$, which is a known property of pseudo MV-algebras.\\
\end{proof}
\section{generalized $\L$ukasiewicz rings}
For a ring $R$ generated by central idempotents, let $Sem(R)=\langle \text{Id}(R), +, \cdot, 0, R\rangle$, where $\text{Id}(R)$ denotes the set of (two-sided) ideals of $R$. It is easily verified that  $Sem(R)=\langle \text{Id}(R), +, \cdot, 0, R\rangle$ is a semi-ring, where $+$ and $\cdot$ are the sum and product of ideals respectively. Define $${}^-,{}^\sim:\text{Id}(R)\to \text{Id}(R)\; \; \text{by}\; \; 
I^-= \{x\in R: Ix=0\}\; \; \text{and}\; \;  I^\sim=\{x\in R: xI=0\}$$
\begin{prop} Given any ring $R$ and $I, J$ ideals of $R$.\\
(i) $I\cdot J=0$ iff $J\subseteq {I}^-$ iff $I\subseteq {J}^\sim $.\\
(ii) $I\subseteq J$ implies $J^-\subseteq I^-$ and $J^\sim\subseteq I^\sim$.\\
(iii) $(I+ J)^-= I^-\cap J^- $;  $(I+ J)^\sim= I^\sim\cap J^\sim $.
$I\subseteq {{I}^\sim}^-$, $I\subseteq {{I}^-}^\sim$.\\
(iv) $I+J \subseteq  {({({I}^\sim \cdot J)}^\sim \cdot {I}^\sim)}^-$ and $I+J\subseteq {( {I}^\sim
 \cdot {({J} \cdot {I}^-)}^\sim)}^-$.
 \\
(v) If $R$ is generated by central idempotents, then $R^-=R^\sim=0.$
\end{prop}
\begin{proof}Let $I, J$ be ideals of $R$.\\
(i)  Follows clearly from the definitions of $\cdot, ^-, ^\sim$.\\
(ii)  Straightforward.\\
(iii)  Since $I, J\subseteq I+J$, it follows that $Ix, Jx \subseteq
(I+J)x$. Thus, $(I+J)x=0$ implies $Ix, Jx=0$ and $(I+ J)^-\supseteq I^-\cap J^- $. In addition, if $Ix=Jx=0$, then $(I+J)x=0$, so $(I+
J)^-\subseteq I^-\cap J^- $. Hence, $(I+ J)^-= I^-\cap J^- $.\\
Similarly, we show that $(I+ J)^\sim= I^\sim\cap J^\sim $.  The inclusions $I\subseteq {({I}^\sim)}^-$ and $I\subseteq {{I}^-}^\sim$ follow from the definitions.\\
(iv) Let  $u \in  {I}^\sim $,   $v \in {({I}^\sim \cdot J)}^\sim $
and $y \in  J $,  we have  $uy \in {I}^\sim \cdot J $ and  $0=
v(uy)=  (vu)y $ and we obtain that $vu  \in  {J}^\sim $. So  $vu \in
 I^\sim\cap J^\sim =(I+ J)^\sim $.
 Since a typical element in  ${({I}^\sim \cdot J)}^\sim \cdot {I}^\sim $ is  a sum of elements of the type $vu $, we obtain that
${({I}^\sim \cdot J)}^\sim \cdot {I}^\sim \subseteq (I+ J)^\sim $
and conclude that $I+J  \subseteq {({I +J})^\sim}^- \subseteq
{({({I}^\sim \cdot J)}^\sim \cdot {I}^\sim)}^- $. The proof that $I+J\subseteq {( {I}^\sim \cdot {({J} \cdot {I}^-)}^\sim)}^-$ is similar to the above.\\
(v) \  Let $x\in R^-$, then $Rx=0$. But since $R$ is generated by central idempotents, there exists $e\in R$ such that $ex=x$. Thus $x\in Rx=0$ and $x=0$.
\end{proof}
\begin{defi}
A ring $R$ is called a generalized $\L$ukasiewicz ring (GLR) if it is generated by central idempotents and for all ideals $I, J$ of $R$, \\
(GLR-1) $I + J=  {({({I}^\sim \cdot J)}^\sim \cdot {I}^\sim)}^- = {({I}^\sim \cdot {({J} \cdot {I}^-)}^\sim)}^- $;\\
(GLR-2) ${({J}^\sim \cdot {I}^\sim)}^- = { ({J}^-\cdot {I}^-)}^\sim$.
\end{defi}
Note that it follows from the definition that if $R$ is a GLR, then $Sem(R)$ is generalized \L ukasiewicz semi-ring. Therefore, by Proposition \ref{AS}, $A(Sem(R))$ is a pseudo MV-algebra, which for simplicity, will be denoted throughout the rest of the paper by $A(R)$. More specifically, $A(R)=\langle \text{Id}(R), \oplus, \odot,^-,
^\sim,0,1\rangle$ is a pseudo MV-algebra, where
$$I \oplus J= {({J}^\sim \cdot {I}^\sim)}^- = { ({J}^-\cdot
{I}^-)}^\sim,$$
 $$I\odot J = I\cdot J= \left\{\sum a_{i}b_{i}:
a_{i}\in I, b_{i}\in J\right\},$$ 
$$I^-= \{x\in R: Ix=0\}, I^\sim=\{x\in
R: xI=0\},$$
$$0=\{0\}, 1=R.$$
We know that the underline lattice $\langle \text{Id}(R), \vee, \wedge, 0, 1\rangle$ of $A(R)$ is distributive, where 
 $I\vee J =  {({({I}^\sim \cdot J)}^\sim \cdot {I}^\sim)}^- = {({I}^\sim \cdot {({J} \cdot {I}^-)}^\sim)}^- = I+J$, $I\wedge J = I\cap J$. \\
Moreover, $\langle \text{Id}(R), \vee, \wedge, 0, 1\rangle$ is a complete lattice. The supremum is given by the sum of ideals and the infimum by the intersection.\\
Since the operation $\oplus$ distributes over $\vee$ in any pseudo MV-algebra, then in $A(R)$ the following identity holds.
 $$I \oplus( J+ K)= I\oplus J + I\oplus K$$
 We have the following crucial lemma.
\begin{lem}\label{mv}
For every generalized \L ukasiewicz ring $R$ the pseudo MV-algebra $A(R)$ is commutative, that is $A(R)$ is an MV-algebra.
\end{lem}
\begin{proof}
Note that as observed above, $A(R)$ is a  complete pseudo MV-algebra, by \cite[Proposition 6.4.14]{DP}, $A(R)$ is Archimedean. In addition, every Archimedean pseudo MV-algebra is commutative \cite[Theorem 4.2]{a}, that is an MV-algebra. Therefore, $A(R)$ is an MV-algebra.
\end{proof}
From Lemma \ref{mv}, the following properties clearly hold in every generalized \L ukasiewicz ring.
\begin{itemize}
\item[(AN)] For every ideal $I$ of $R$, $I^-=I^\sim$, which we shall denote simply by $I^\ast$. 
\item[(CO)] For all ideals $I, J$ of $R$, $I\cdot J=J\cdot I$. 
\item[(LR)] For all ideals $I, J$ of $R$, $I+J=(I^\ast\cdot (I^\ast\cdot J)^\ast)^\ast$
\end{itemize}
Indeed, it is also easy to see that a ring that is generated by central idempotents and satisfying, (AN), (CO), (LR) is a GLR.
\begin{prop} A ring is a GLR if and only if it is generated by central idempotents and satisfies, (AN), (CO), and (LR).
\end{prop}
Since the identity $x^{\ast\ast}=x$ holds in every MV-algebra, we have the following.
\begin{prop}\label{inv}
For every ideal $I$ of a generalized \L ukasiewicz ring $R$, $I^{\ast\ast}=I$.
\end{prop}
It follows that every GLR is a dual ring. Details on dual rings can be found in \cite{HN, kap}.
Note that commutative GLRs are exactly the \L ukasiewicz rings as treated in \cite{BN}. In addition to these, we present an example of a non-commutative GLR.
\begin{ex}\label{mainex}
Let $R$ be any unitary \L ukasiewicz ring and $n\geq 1$ be any integer. Then the ring $M_n(R)$ of $n\times n$ matrices over $R$ is a GLR. To see this, first we recall for any unitary ring $R$ (see for e.g \cite[Theorem 3.1]{tyl}), the ideals of $M_n(R)$ are of the form $M_n(I)$, where $I$ is an ideal of $R$. In addition, a simple verification reveals the following axioms for any ring $R$.
$$M_n(I)^\sim=M_n(I^\sim), M_n(I+J)=M_n(I)+M_n(J)$$$$M_n(I)^-=M_n(I^-), M_n(I\cdot J)=M_n(I)\cdot M_n(J)$$
It is therefore clear that if $R$ is a unitary \L ukasiewicz ring (or even a unitary GLR), that $M_n(R)$ is generated by central idempotents (is indeed unitary) and satisfies (AN), (CO), and (LR).
\end{ex}
We would like to describe the relationship between ideals of $R$ and those of $Sem(R)$, when $R$ is a GLR. For the remainder of this section, $R$ will denote a GLR and $S$ its associated semi-ring, that is $S=Sem(R)$. Note that since $R$ is generated by central idempotents, for every $x\in R$, $RxR:=\{\sum_{i=1}^n r_ixs_i: n\geq 1, r_i, s_i\in R\}$ is the ideal of $R$ generated by $x$.\par
For every ideal $I$ of $S$, we define
$$S(I):=\{J\in \text{Id}(R): J\subseteq Rx_{1}R + Rx_{2}R+...+Rx_{n}R, \; \text{for some}\;   x_{1},...,x_{n}\in I\}$$
Then, by Proposition \ref{equiv-ideal}, it is straightforward that $S(I)$ is an ideal of $S$ and
$$S(I)=\left\{\sum_{i=1}^n J_ix_iL_i:n\geq 1, J_i, L_i\in \text{Id}(R), x_i\in I\right\} $$
Indeed, $S(I)$ is the ideal of $S$ generated by $X:=\{RxR\in \text{Id}(R):x\in I\}$.\\
One should also observe that if $I$ is proper, so is $S(I)$. Indeed, if $R\in S(I)$, then there are $x_{1},
x_{2},...,x_{n}\in I$ such that $R\subseteq Rx_{1}R +
Rx_{2}R+...+Rx_{n}R\subseteq I$  and so $I=R$.\par
To reverse the construction above, we define
$S^{-1}(\textbf{I}):= \{x\in R : RxR \in \textbf{I} \}$,
for each ideal $\textbf{I}$ of $S$.
\begin{prop}\label{AS3} (i)For each ideal $\textbf{I}$ of $S$, $S^{-1}(\textbf{I})$ is an ideal of $R$.\\
(ii) For each ideal $\textbf{I}$ of $S$, $S(S^{-1}(\textbf{I}))\subseteq \textbf{I}$.\\
(iii)\ If $\textbf{I}\subseteq \textbf{J}$, then  $
S^{-1}(\textbf{I})\subseteq S^{-1}(\textbf{J})$.
\end{prop}
\begin{proof}
(i)\ Assume that $\textbf{I}$ is an ideal of $S$ and let $I =
S^{-1}(\textbf{I})$. Let $x,y \in I$, we have $RxR \in \textbf{I}$
and $RyR \in \textbf{I}$. Since $\textbf{I}\in Id(S)$, we have $RxR+
RyR \in \textbf{I}$.  From the fact that  $R(x+ y)R \subseteq RxR+
RyR$, we deduce that $R(x+ y)R \in \textbf{I}$ and $x+y\in I$.\par In addition, let
$x\in I$ and $y\in R$.   Since  $\textbf{I}\in Id(S)$, $RxR \in
\textbf{I}$ and $RyR \in \text{Id}(R) $, it follows that
 $RxR\cdot RyR \in \textbf{I}$.  From this and the fact that $RxyR
\subseteq RxR\cdot RyR$, we have $RxyR \in \textbf{I}$ as $\textbf{I}$ is an ideal of $S$. That is  $xy \in I$. A similar argument shows that $yx\in I$. Thus, $S^{-1}(\textbf{I})$ is an
ideal of $R$.\\
(ii) Let $J\in S(S^{-1}(\textbf{I}))$, then $J\subseteq Rx_{1}R+Rx_{2}R+...+Rx_{n}R $, for  some
$x_{1},...,x_{n}\in S^{-1}(\textbf{I})$. But $x_{i}\in S^{-1}(\textbf{I})$ means that $Rx_{i}R \in \textbf{I}$ and then,
$J\subseteq Rx_{1}R+Rx_{2}R+...+Rx_{n}R \in \textbf{I}$. Hence $J
\in \textbf{I}$.\\
(iii) Clear.
\end{proof}
\begin{prop}\label{AS4}  For every ideal $I$ of $R$,  $I = S^{-1}(S(I))$.
\end{prop}
\begin{proof}
\ Assume that  $I$ is an ideal of $R$ and $x\in I$. It is clear
that $RxR\in S(I)$ and then $x\in S^{-1}(S(I))$. Conversely, let
$x\in S^{-1}(S(I))$. Thus $RxR\subseteq Rx_{1}R+Rx_{2}R+...+Rx_{n}R$,
for some $x_{1},...,x_{n}\in I$. In particular, since each $Rx_iR\subseteq I$, then $x\in RxR\subseteq I$, and $x\in I$. Thus, $I = S^{-1}(S(I))$.\\
\end{proof}
For the next result, $\text{FG}(R)$ denotes the set of finitely generated ideals of $R$.
\begin{prop}\label{AS5} For each ideal $\textbf{I}$ of $S$,\\
(i)\  $\textbf{I}\cap \text{FG}(R) \subseteq S(S^{-1}(\textbf{I}))$.\\
(ii)\ If every ideal of $\textbf{I}$ is finitely generated, then $\textbf{I}=
S(S^{-1}(\textbf{I}))$.
\end{prop}
\begin{proof}
 (i)\ Suppose $J \in \textbf{I}$ and $J= Rx_{1}R+Rx_{2}R+...+Rx_{n}R$,
 for  some $x_{1},...,x_{n}\in J$. Since $Rx_iR\subseteq J$ for all $i$, and $J\in \textbf{I}$, which is an ideal of $S$, then $Rx_iR\in  \textbf{I}$ for all $i$. That is $x_i\in S^{-1}(\textbf{I})$ for all $i$, and $J\in S(S^{-1}(\textbf{I}))$. Hence,  $\textbf{I}\cap \text{FG}(R) \subseteq S(S^{-1}(\textbf{I}))$ as needed.\\
(ii) By assumption, $\textbf{I}\subseteq FG(R)$, hence $\textbf{I}=\textbf{I}\cap \text{FG}(R) \subseteq S(S^{-1}(\textbf{I}))$. The equality is obtained by combining the above with Proposition \ref{AS3}(ii).
\end{proof}
Note all ideals of Noetherian rings are finitely generated. Therefore, if $R$ is Noetherian, then $\textbf{I}=S(S^{-1}(\textbf{I}))$. Thus,
 there is a one-to-one correspondence between the ideals of $R$ of those of $S$.\\
Given an ideal $I$ of $R$, we consider the ring $R/I$. Then ideals of  $R/I$ are of the form $J/I := \{x/I: x\in J\}$ where $J$ is an ideal
of $R$ such that $I\subseteq J$. For ideals $J,K$ of $R$ such that
$I\subseteq J$ and $I\subseteq K$, we have $J/I+ K/I= (J+K)/I$
and $(J/I)\cdot (K/I)= (J\cdot K)/I$.
\begin{prop}\label{qneed2}Let $R$ be any ring, and $I,J$ be
ideals of $R$ such that $I\subseteq J$. Then,
\begin{itemize}
\item[(i)]  $I\subseteq (I^\sim \cdot J)^-$ and $I\subseteq (J\cdot I^-)^\sim$.
\item[(ii)]  $(J/I)^-= (I^\sim \cdot J)^-/I$ and $(J/I)^\sim= (J\cdot I^-)^\sim/I$.
\item[(iii)] If $R$ is a GLR, then $(J/I)^\ast=(I^\ast\cdot J)^\ast/I$.
\end{itemize}
\end{prop}
\begin{proof}
(i) \ We have $ I^\sim \cdot J\subseteq I^\sim$ and then  $I\subseteq 
{{I}^\sim}^-\subseteq (I^\sim \cdot J)^-$.\\ Similarly, $(J\cdot I^-)^\sim\subseteq I^-$ and then $I\subseteq 
{{I}^-}^\sim\subseteq (J\cdot I^-)^\sim$.\\
(ii)We have $$ \begin{aligned}
(J/I)^-&= \{x/I \in R/I: (y/I)(x/I)= 0 \ \text{for  all}\; \;   y\in J \}\\
&=\{x/I \in R/I: yx\in I \ \text{for  all}\; \;  y\in J \}\\ 
&=\{x/I \in R/I: Jx\subseteq I \}\\ 
&=\{x/I \in R/I: I^\sim \cdot Jx\subseteq I^\sim\cdot I=0 \}\\ 
&=\{x/I \in R/I: I^\sim \cdot Jx=0 \} \\ 
&=\{x/I \in R/I: x\in (I^\sim \cdot J)^- \}\\
&= (I^\sim \cdot J)^-/I
\end{aligned}  
$$
A similar argument shows that $(J/I)^\sim= (J\cdot I^-)^\sim/I$.\\
(iii) Since $A(R)$ is a GLR, then the conclusion holds by (ii) and (AN).
\end{proof}
The next result shows that GLRs are closed under finite products, arbitrary direct sums, and quotients.
\begin{thm}\label{oper}
 \begin{itemize}
 \item[(1)] Any finite direct product of GLRs is again a GLR.
 \item[(2)] Any quotient of a GLR by a proper ideal is again a GLR.
 \item[(3)] Any direct sum of GLRs is again a GLR.
  \end{itemize}
\end{thm}
\begin{proof} First, it is clear that finite direct products, quotients, direct sums of rings generated by central idempotents and rings generated by central idempotents.\\
(1) Suppose that $R=\prod^n_{i=1}R_i$, with each $R_i$ a GLR, then $A(R_i)$ is an MV-algebra by Lemma \ref{mv}. On the other hand, every ideal of $R$ has the form $I=\prod^n_{i=1}I_i$, where $I_i$ is an ideal of $R_i$. Now, it is clear that (AN), (CO) hold, and the proof of $(LR)$ is the same as in \cite{BN}.\\
(2) Suppose that $R$ is a GLR and $I$ is an ideal of $R$, then since $R$ satisfies (AN), (CO), it follows from Proposition \ref{qneed2} that $R/I$ satisfies (AN), (CO). The proof that $R/I$ satisfies (LR) is the same as that of \cite[Proposition 3.9]{BN} given that (iii) of Proposition \ref{qneed2} holds.\\
(3) As above, this proof can be adapted from that of \cite[Proposition 3.12]{BN} using the conditions (AN), (CO), (LR).
\end{proof}
\begin{rem}
An infinite (direct) product of GLRs needs not be a GLR. Indeed, consider $R=\prod_{n=1}^\infty F$, where $F$ is a field. We claim that $R$ is not a GLR. To see this, consider $I=\{(x_n):x_{2n}=0\}$, $J=\oplus_{n=1}^\infty F$ and $K=\{(x_n):x_{2n+1}=0\}$, which are all ideals of $R$. One can verify that $$I^\ast=K, (I+J)^\ast=0, {({I}^\ast \cdot J)}^\ast \cdot {I}^\ast=K$$
Thus, $(I+J)^\ast\ne {({I}^\ast \cdot J)}^\ast \cdot {I}^\ast$ and (LR) fails.
\end{rem}
\begin{prop}
(i)\ For all $I, J, K \in \text{Id}(R)$, $$I\cap (J+K)= I\cap J +
I\cap K.$$
(ii)\ If $\{J_{i}\}_{i\in T}$ is a family of ideals of  $R$, then
$$I+\bigcap_{i\in T}J_{i}= \bigcap_{i\in T}(I+J_{i}).$$
\end{prop}
\begin{proof}
(i)\ Since $\langle \text{Id}(R), \wedge,\vee,0,R\rangle$ is a distributive lattice and $\wedge = \cap, \vee=+$, we have $I\cap (J+K)= I\wedge (J\vee K)= (I\wedge J)\vee (I\wedge K)= (I\cap J)+ (I\cap K)$.\\
(ii)\ Let  $\{J_{i}\}_{i\in T}$ be a family of ideals of   $R$, since
$A(R)=\langle \text{Id}(R), \oplus,^\ast,
0,1\rangle$ is a complete MV-algebra, we have
$I+\bigcap_{i\in T}J_{i}=I\vee\bigwedge_{i\in T}J_{i} =\bigwedge_{i\in T}(I\vee
J_{i})= \bigcap_{i\in T}(I+ J_{i})$.
\end{proof}
\begin{prop}\label{AS1}
Prime ideals of GLRs are maximal.
\end{prop}
\begin{proof}
Let $R$ be a GLR. If $P$ is a prime
ideal of $R$, then $R/P$ is a prime ring (see for e.g., \cite{tyl}, observation following definition 10.15). Let $J$ be a proper ideal of
$R$, $P \subseteq J$. Suppose that $J \neq P$. Since the left and right annihilators of any nonzero ideal of a prime ring are zero, then $(J/P)^{\sim}= (J/P)^{-}=0$. But, since $R/P$ is a GLR by Theorem \ref{oper}(2), we conclude that $J/P=((J/P)^{\ast})^{\ast}= 0^{\ast}=R/P$. Hence, $J = R$ and $P$ is maximal.
\end{proof}
Let $R$  be a ring and $M$ a subring (with or without a unity). For every ideal $I$ of $R$, let
  $$ I^{-M} = \{x\in M: Ix=0\}\; \;
\text{and}\; \;  I^{\sim M}  =\{x\in M : xI=0\}.$$
Observe that $I^{-M}=M\cap I^-$ and $I^{\sim M}=M\cap I^\sim$. 
\begin{prop}
Let $R$  be a GLR and $M$ an ideal of $ R $  such that  $M\cap M^- =
0$ and $M\cap M^\sim = 0$. Then $M$ is a GLR.
\end{prop}
\begin{proof}
Since $R$ satisfies (AN), then $M^-=M^\sim=M^\ast$. Thus, $M\cap M^\ast =0$, which implies $M + M^\ast =M\vee M^\ast=M\oplus M^\ast=R $. It follows that any ideal of $ M $ is an ideal of  $R$, i.e., $\text{Id}(M)\subseteq \text{Id}(R)$.\\
 Let  $ x\in M$, then as $R$ is generated by central idempotents, there is $e\in R$ central idempotent such that $ex=x.$
 Since  $M + M^\ast =R$, there are $m\in M$, $m' \in M^\ast$ such that
  $m + m' =e $. Then, $me+m'e=e^2=e$, so $m+m'=me+m'e$. Thus, $m-me=m'-m'e\in M\cap M^\ast=0$ and it follows that $me=m$. A similar argument shows that $m=em$. On the other hand, from $m + m' =e $, one gets $m^2=me=m$. Thus, $m$ is idempotent in $M$, and clearly satisfies $mx=ex=x$. To see that $m$ is central in $M$, let $a\in M$. Then since $e$ is central in $R$, then $ea=ae$, that is $(m+m')a=a(m+m')$. Hence, $ma=am$ and $m$ is central in $M$. Therefore, $M$ is generated by central idempotents. \\
In addition since $\text{Id}(M)\subseteq \text{Id}(R)$, then $M$ satisfies (AN) and (CO). Finally, $M$ satisfies (LR) and the proof of this is identical to the one given in \cite[Proposition 6.1]{BN}. Hence, $M$ is a GLR as claimed.
\end{proof}
One should observe that since $\text{Id}(M)\subseteq \text{Id}(R)$, the proof of (LR) shows that in $A(M)$ , $ I \oplus_{M} J =
  I\oplus J $.
\section{Representation of GLRs}
The goal of this section is to find a representation of GLRs that would generalize that of \L ukasiewicz rings obtained in \cite[Theorem 7.7]{BN}.\par
Recall that a ring $R$ is called a special primary ring (\textit{SPIR}) if $R$ has a unique maximal ideal $M$ and every proper ideal in $R$ is a power of $M$. By chain ring, we mean a ring whose ideals are linearly ordered by inclusion.\par
We have the following class of GLRs.
\begin{prop}\label{acr}
Every unitary special primary ring is a GLR.
\end{prop}
\begin{proof}
$R$ has a unique maximal ideal $M$, which is nilpotent. Therefore there exists a natural number $n$ such that $M^{n}=0$ (we assume that $n$ is the smallest integer for which $M^n=0$). It follows that for all $1\leq i<j\leq n$, $M^i\ne M^j$. \\
We prove that for every $1\leq k< n$, $M^{k\sim}=M^{n-k}$. Note that $M^{n-k}\cdot M^k=M^n=0$, so $M^{n-k}\subseteq M^{k\sim}$. But since every ideal of $R$ is a power of $M$, then $M^{k\sim}=M^i$ where $i\leq n-k$. In addition $0=M^{k\sim}\cdot M^k=M^i\cdot M^k=M^{i+k}$, thus by the minimality of $n$, we may conclude that $n\leq i+k$. It follows that $n-k\leq i$, and $i=n-k$. Hence, $M^{k\sim}=M^{n-k}$ as needed. A similar argument shows that $M^{k-}=M^{n-k}$. In particular, $I^\sim=I^-$ for all ideals of $R$, which is (AN). In addition, it also follows that $R$ satisfies (CO).\\
 To prove (LR), let $I=M^i$, $J=M^j$ with $1\leq i, j\leq n$. We consider cases.\\
Case 1: Suppose $i\leq j$. If $j=n$, then $J=0$ and the identity holds trivially. So, assume that $j<n$, then $I^\sim=M^{n-i}$ and $J^\sim=M^{n-j}$. Note that $I+J=I$, and since $n-i+j\geq n$, then $I^\ast \cdot J=0$. It follows easily that $I+J={({({I}^\ast \cdot J)}^\ast \cdot {I}^\ast)}^\ast$.\\
Case 2: Suppose that $j<i$, then $I+J=J$. If $i=n$, then $I=0$ and the identities are obvious. Assume that $i<n$, then $I^\ast=M^{n-i}$ and $I^\ast\cdot J=M^{n-i+j}$ (note that $n-i+j<n$). Hence, $(I^\ast\cdot J)^\ast=M^{i-j}$,  $((I^\ast\cdot J)^\ast)\cdot I^\ast=M^{n-j}$, and $(((I^\ast\cdot J)^\ast)\cdot I^\ast)^\ast=M^j=J=I+J$.\\
This completes the proofs of (AN), (CO) and (LR). Thus, $R$ is a GLR as claimed.
\end{proof}
\begin{prop} \label{artch}
Every unitary left Artinian ring whose left ideals are linearly ordered by inclusion is a GLR.
\end{prop}
\begin{proof}
It is known that if $R$ is left Artinian, then $J(R)=B(R)$, where $B(R)$ is the Brown-McCoy radical of $R$ (the intersection of maximal two-sided ideals of $R$). Since the left ideals of $R$ are linearly ordered, then $R$ has a unique maximal (left) ideal $M$. Hence $J(R)=M$, which is a nilpotent two-sided ideal of $R$ (see for e.g., \cite[Theorem 4.12]{tyl}), then there exists a natural number $n$ such that $M^{n}=0$ (we assume that $n$ is the smallest integer for which $M^n=0$). It follows that for all $1\leq i<j\leq n$, $M^i\ne M^j$. We shall prove that every ideal of $R$ is a power of $M$. Let $I$ be a left ideal of $R$, and let $k$ be the smallest integer such that $M^k\subseteq I$. Then $M^{k-1}\nsubseteq I$ and since the left ideals of $R$ are linearly ordered, we conclude that $M^k\subseteq I\subsetneq M^{k-1}$. Now we consider, the ring $R/M$ which is semisimple (in fact simple), and since $M^{k-1}/M^k$ is an $R/M$-module, then $M^{k-1}/M^k$ is semisimple. Thus $M^{k-1}/M^k$ is a direct sum of simple $R/M$-submodules. But, $M^{k-1}/M^k$ is also a chain module (its submodules are linearly ordered). Therefore, $M^{k-1}/M^k$ is a simple $R/M$-module. Since $I/M^k$ is a proper $R/M$-submodule of $M^{k-1}/M^k$, it follows that $I=M^k$. Therefore, every left ideal of $R$ is a power of $M$. In particular, every ideal of $R$ is a power of $M$ and $R$ is a special primary ring. The conclusion now follows from Proposition \ref{acr}.
\end{proof}
One should observed that unitary special primary rings are quasi-Frobenius rings (QF-rings). \par
We have the following nice examples of unitary special primary rings.	
\begin{ex}
Let $F$ be a field and let $T=F[x]/\langle x^n\rangle$, where $n\geq 1$ is a fixed integer. Then $T$ is a \L ukasiewiecz ring (see \cite[Ex. 3.5]{BN}). In fact $T$ is a special primary ring, and using the properties stated in Example \ref{mainex}, one can easily verify that $R=M_m(T)$ is a unitary special primary ring for all integer $m\geq 1$.
\end{ex}
\begin{ex}
Let $D$ be a Noetherian Dubrovin valuation ring, then the left ideals of $D$ are known to be linearly ordered by inclusion (see for e.g., \cite{Wa}). Let $M$ be the unique maximal (left) ideal of $D$ and $R=D/M^n$, where $n\geq 1$ is an integer. We claim that $R$ is a special primary ring. That the left ideals of $R$ are linearly ordered by inclusion follows from the fact that the same is true for $D$. In addition, $J(R)=M/M^n$, and from this it follows that $J(R)^n=0$ and $J(R)$ is nilpotent. Finally $R/J(R)\cong D/M$, which is semisimple (in fact simple). Hence, by Hopkins-Levitzki Theorem, $R$ is Artinian and it follows from Proposition \ref{artch} that $R$ is a GLR.
\end{ex}
\begin{prop}\label{zorn}
Every proper ideal of a GLR is contained in a maximal ideal.
\end{prop}
\begin{proof}
Let $I$ be a proper ideal of a generalized \L ukasiewicz ring $R$, then there exists $x\in R$, with $x\notin I$. Since $R$ is generated by central idempotents, let $r\in R$ be a central idempotent element such that $xr=x$. Then for every integer $n\geq 1$, $xr^n=x$. Thus $r^n\notin I$ for all integers $n\geq 1$, that is $I\cap \{r^n:n\geq 1\}=\emptyset$. The family $\mathcal{I}$ of ideals of $R$ containing $I$ and disjoint with $\{r^n:n\geq 1\}$ is an inductive family under inclusion. By Zorn's lemma, this family has a maximal element $P$. We claim that $P$ is a prime ideal of $R$. Clearly $P$ is proper since $r\notin P$. In addition suppose that $RaR\cdot RbR\subseteq P$ and $a\notin P$, $b\notin P$. Then, by the maximality of $P$, $P+RaR$ and $P+RbR$ are not members of $\mathcal{I}$. But, since both ideals contain $I$, then there exist $n, m\geq 1$ such that $r^n\in P+RaR$  and $r^m\in P+RbR$. Thus, $r^{m+n}\in (P+RaR)\cdot (P+RbR)=P+RaR\cdot RbR\subseteq P$. Hence, $r^{m+n}\in P$, which is a contradiction since $P\in \mathcal{I}$. Therefore, $P$ is a prime ideal of $R$ that contains $I$ and by Proposition \ref{AS1}, $P$ is a maximal ideal of $R$.
\end{proof}
\begin{cor}
For every generalized \L ukasiewicz ring $R$, the pseudo MV-algebra $A(R)$ is complete and atomic.
\end{cor}
\begin{proof}
First, we recall that the ideals of any ring form a complete lattice. Suppose that $R$ is a GLR, and let $0\ne I\in A(R)$, then $I^\ast\ne R$ and by Proposition \ref{zorn}, there exists a maximal ideal $M$ of $R$ such that $I^\ast\subseteq M$. Thus, $M^\ast \subseteq I$. Clearly, $M^\ast$ is an atom in $A(R)$. 
\end{proof}
On should observe that if $R$ is a GLR, then $M\mapsto M^\ast$ is a one-to-one correspondence between the maximal ideals of $R$ and the atom of $A(R)$.
\begin{prop}\label{finite}
For every unitary generalized \L ukasiewicz ring $R$, the MV-algebra $A(R)$ is finite.
\end{prop}
\begin{proof}
Let $R$ be a unitary GLR. Then by Proposition \ref{inv}, $R$ is a dual ring and by \cite[Theorem 3.4]{HN}, $R/J$ (where $J$ is the Jacobson radical of $R$) is a semi-simple Artinian ring. By the classical Wedderburn-Artin Theorem, $R/J\cong M_{n_1}(D_1)\times\cdots \times M_{n_r}(D_r)$, for some integers $n_1,\cdots, n_r\geq 1$ and some division rings $D_1, \cdots, D_r$. It follows that $R/J$ has finitely many maximal ideals. But $R/J$ and $R$ have the same number of maximal ideals. Therefore, $R$ has finitely many maximal ideals. Hence, by the observation above, the complete and atomic MV-algebra $A(R)$ has finitely many atoms. It follows from \cite[Corollary 6.8.3]{C2} that $A(R)$ is a finite product of finite MV-chains. Thus, $A(R)$ is finite.
\end{proof}
\begin{prop}\label{unitary}
Let $R$ be a GLR such that $A(R)$ is a finite MV-chain, then $R$ is unitary.
\end{prop}
\begin{proof}
We may assume that $R$ is non-trivial since the result is clear if $R$ is the trivial ring. Since $R$ is generated by central idempotents, then $R$ contains a non-zero central idempotent element $e$. Consider $eR$, which is a non-zero idempotent ideal of $R$. That is, $eR$ is a non-zero idempotent element in the MV-chain $A(R)$. But, the only non-zero idempotent ($x\odot x=x$) in any finite MV-chain is $x=1$. Thus, $eR=R$ and $e$ becomes the unit of $R$.
\end{proof}
We shall now set up the ground to prove the representation theorem for GLRs. Recall that for every generalized \L ukasiewicz ring $R$, $A(R)$ is complete and atomic. Therefore, by \cite[Corollary 6.8.3]{C2}, there exists a set $X$ and integers $n_x\geq 1$ such that $A(R)\cong \prod_{x\in X}\L _{n_x}$, where $\L_n$ denotes the \L ukasiewicz chain with $n$ elements.. Let $\varphi:A(R)\to A:=\prod_{x\in X}\L _{n_x}$ be an isomorphism. For each $x\in X$, let $a_x\in A$ defined by $a_x(x)=1$ and $a_x(y)=0$ for all $y\ne x$. Let $R_x=\varphi^{-1}(a_x)$, then since $\varphi$ is a homomorphism, $\varphi(R^\ast_x)=a^\ast_x$. Note that in $A$, $a_x\vee a^\ast_x=1$, $a_x\wedge a^\ast_x=0$, $a_x^2=a_x$ for all $x\in X$ and $a_x\wedge a_y=0$ for $x\ne y$. Thus, the following facts easily follow from the fact that $\varphi$ is an isomorphism.
\begin{itemize}
\item[Fact 1]:  $R_x+R^\ast_x=R$ for all $x\in X$;
\item[Fact 2]: $R_x\cap R^\ast_x=0$ for all $x\in X$;
\item[Fact 3]:  $R_x\cap R_y=0$ for all $x\ne y$;
\item[Fact 4]: $R_x$ is idempotent, i.e., $R_x^2=R_x$.
\end{itemize}
We will also need the following lemma, that provides the main building blocks of GLRs.
\begin{lem}\label{acr2}
For every $x\in X$, 
\begin{itemize}
\item[1.] The MV-algebras $A(R/R^\ast_x)$ and $\L_{n_x}$ are isomorphic. 
\item[2.] $R/R^\ast_x$ is a unitary special primary ring.
\item[3.] $R_x\cong R/R^\ast_x$ as rings. In particular, $R_x$ is a unitary special primary ring.
\end{itemize}
\end{lem}
\begin{proof}
1. Note that by  \cite[Proposition 6.4.3(i)]{C2}, $A/(a^\ast_x]\cong (a_x]\cong \L_{n_x}$. So, it is enough to prove that $A(R/R^\ast_x)\cong (a_x]$. But, one can verify that $a\mapsto (\varphi^{-1}(a)+R^\ast_x)/R^\ast_x$ is an isomorphism from $(a_x]\to A(R/R^\ast_x)$. \\
2. First, $R/R^\ast$ is a GLR by Proposition \ref{oper}. Since $A(R/R^\ast_x)\cong \L_{n_x}$, then $R/R^\ast_x$ is unitary by Proposition \ref{unitary} and has a unique maximal two-sided ideal. More precisely, if $\psi: A(R/R^\ast_x)\to \L_{n_x}$ is an isomorphism, then $M_x:=\psi^{-1}\left((n_x-2)/(n_x-1)\right)$ is the unique maximal two-sided ideal of $R/R^\ast_x$. Now, let $I$ be an ideal of $R/R^\ast_x$, then $\psi(I)=k_x/(n_x-1)$ or some $0\leq k_x\leq n_x-1$. But, since $k_x/(n_x-1)=\left((n_x-2)/(n_x-1)\right)^{n_x-k_x-1}$, then $I=M_x^{n_x-k_x-1}$. Thus, $R/R^\ast_x$ is a special primary ring as claimed. \\
3. Consider $\psi: R_x\to R/R^\ast_x$, the composition of the inclusion $R_x\hookrightarrow  R$ followed by the natural projection $R\to R/R_x$. That is, $\psi(t)=t+R^\ast_x$ for all $t\in R_x$. Then $\psi$ is clearly a ring homomorphism and since $R_x\cap R^\ast_x=0$, then $\psi$ is injective. In addition $\psi$ is onto because $R_x+R^\ast_x=R$. Thus, $\psi$ is an isomorphism as needed.
\end{proof}
\begin{thm}\label{main}
A ring $R$ is a GLR if and only if $R$ is isomorphic to a direct sum of unitary special primary rings.
\end{thm}
\begin{proof}
We know that any unitary special primary ring is a GLR (Proposition \ref{acr}), and the direct sum of GLRs is again a GLR (Proposition \ref{oper}(3)). It remains to prove that every GLR is isomorphic to a direct sum of unitary special primary rings. We shall prove that if $R$ is a GLR, with the notations above 
 $$R\cong \displaystyle\bigoplus_{x\in X}R/R^\ast_x$$
Note that by Lemma \ref{acr2}, $R_x$ is a unitary special primary ring for all $x\in X$. Let $R'=\sum_{x\in X}R_x$, be the ideals sum. Then $R$ is an ideal of $R$. In addition since, the identity $a\wedge \bigvee_{i}a_i=\bigvee_i(a\wedge a_i)$ holds in any MV-algebra, then $R_y\cap \sum_{x\ne y}R_x=0$. Therefore, the ideal sum is direct. We claim that $R=R'$. To see this, we show that $R'^\ast=0$. So, let $t\in R'^\ast$, then $t\in R^\ast_x$ for all $x\in X$. Thus, $RtR\subseteq R^\ast_x$ and $\varphi(RtR)\leq a^\ast_x$ for all $x\in X$. Hence, $\varphi(RtR)(x)=0$ for all $x\in X$. Whence, $\varphi(RtR)=0$ and since $\varphi$ is an isomorphism, $RtR=0$ and $t=0$. Therefore, $R=R'$ and the isomorphism sought follows from Lemma \ref{acr2}.
\end{proof}
As a consequence of Theorem \ref{main}, we can strengthen Proposition \ref{unitary}.
\begin{cor}
Every GLR with finitely many ideals is isomorphic to a finite direct product of unitary special primary rings. In particular, any such ring is unitary.
\end{cor}
\section{Final remarks}
MV-algebras have both commutative and non-commutative generalizations. One popular commutative generalization of MV-algebra is the notion of BL-algebra, that was introduced by H\'ajek \cite{H2}. Our future goal is to study commutative rings $R$ for which $A(R)$ is a BL-algebra (BL-rings). These rings are closely related to multiplication rings as studied by M. Griffin \cite{Grif}.


\begin{thebibliography}{aaaaa}
\bibitem{BN} L. P. Belluce, A. Di Nola, Commutative rings whose ideals form an MV-algebra, Math. Log. Quart. {\bf 55}(5) (2009)468-486.
\bibitem{CK} I. Chajda, J. Kh\"ur GMV-algebras and meet-semilattices with sectionally antitone permutations, Math. Slovaca, {\bf 56}(2006)275-288.
 \bibitem{CM} R. Cignoli, D. Mundici, An elementary presentation of the equivalence between MV-algebras and $\ell$-groups with strong units, Studia Logica, special issue on Many-valued logics {\bf 61}(1998) 49-64.
 \bibitem{C2} R. Cignoli, I. D'Ottaviano, D. Mundici, Algebraic foundations of many-valued reasoning,
Kluwer Academic, Dordrecht (2000).
 \bibitem{ad} A. Dvure\v{c}enskij, On pseudo MV-algebras, Soft Comput. {\bf 5}(5) (2001)347-354.
\bibitem{a} A. Dvure\v{c}enskij, Pseudo MV-algebras are intervals in $\ell$-groups. J. Aust. Math. Soc.  {\bf 72} (2002), 427-445.
\bibitem{DP} A. Dvure\v{c}enskij, S. Pulmannov\'{a}, New trends in quantum structures, Kluwer, Dordrecht and
Ister Science, Bratislava, (2000).
\bibitem{g1} G. Georgescu, A. Iorgulescu, Pseudo-MV algebras: a non-commutative extension of MV algebras. Inf. tech. (1999)961-968.
\bibitem{g} G. Georgescu, A. Iorgulescu, Pseudo-MV algebras I, Mult.-Valued Log. {\bf 6} (2001), 95-135.
\bibitem{HN} C. R. Hajarnavis, N. C. Norton, On dual ring and their modules, J. Algebra {\bf 93}(1985)253-266.
\bibitem{H2} P. H\'ajek, Metamathematics of fuzzy logic,  Kluwer Academic Publishers, Dordrecht (1998).
\bibitem{Grif} M. Griffin, Multiplication rings via their total quotient rings, Can. J. Math. {\bf 26}(1974)430-449.
\bibitem{kap} I. Kaplansky, Dual rings, Ann. Math. {\bf 49}  (1948)689-701
\bibitem{tyl} T. Y. Lam, A first course in non-commutative rings, Springer-Verlag, New York, 1991.
\bibitem{Wa} A. R. Wadsworth, Dubrovin valuation rings, Perspect. in ring theory (1987)359-374
\end{thebibliography}
\end{document}